\documentclass[11pt]{article}
\usepackage{geometry}                		
\usepackage[parfill]{parskip}    		
\usepackage{amssymb}
\usepackage{amsmath}

\usepackage{palatino}
\usepackage{textcomp}
\usepackage{graphicx}
\usepackage{epsfig}
\usepackage{epstopdf}
\usepackage{mathrsfs}
\usepackage[english]{babel}
\usepackage{amscd}
\usepackage{amsthm}

\usepackage{fancyhdr} 
\newtheorem{theorem}{Theorem}[section]
\newtheorem{lemma}[theorem]{Lemma}

\title{A lower bound on the probability that a binomial random variable is exceeding its mean}

\author{
Christos Pelekis\thanks{Informatics Section, KU Leuven,
Celestijnenlaan 200A,
3001, Belgium; Email: pelekis.chr@gmail.com}  \quad \quad Jan Ramon\thanks{Informatics Section, KU Leuven,
Celestijnenlaan 200A,
3001, Belgium; Email: Jan.Ramon@cs.kuleuven.be} }

\begin{document}

\maketitle

\begin{abstract}  We provide a lower bound on the probability that a binomial random variable is exceeding its mean. 
Our proof employs estimates on the mean absolute deviation and the tail 
conditional expectation of binomial random variables. 
\end{abstract}

\noindent {\emph{Keywords}:  lower bounds; binomial tails; tail conditional expectation; mean absolute deviation;
hazard rate order

\section{Prologue, related work and main result}

Given a positive integer $n$ and a real number $p\in (0,1)$, we denote by $\text{Bin}(n,p)$ a binomial 
random variable of parameters $n$ and $p$.  Here and later,  given two random variables $X,Y$, 
the notation $X\sim Y$ will indicate that $X$ and $Y$ have the same distribution.  
The main purpose of this note is to illustrate that estimates on the mean absolute deviation 
of a binomial random variable yield  
a lower bound on $\mathbb{P}\left[X\geq np\right]$, where $X\sim \text{Bin}(n,p)$. 
It should come as no surprise that there exists general machinery that can be 
employed to such a problem. 
For example, using Cauchy-Schwartz inequality one can show that, for any random variable $Z$  
whose mean equals zero, it holds 
\begin{equation}\label{markver}
\mathbb{P}\left[Z\geq 0\right] \geq  \frac{1}{4}\cdot \frac{  \{\mathbb{E}\left[|Z|\right]\}^2 }{\mathbb{E}\left[ Z^2 \right]}   \end{equation} 
and the bound can be improved further under information on higher moments (see Veraar \cite{Veraar}).  
If we now let $Z=X-np$, where $X\sim \text{Bin}(n,p)$, then (\ref{markver}) provides a lower bound 
on the probability that a binomial random variable is exceeding its expectation that is 
expressed in terms of the second moment of $X$ and 
(the square of) its mean absolute deviation, i.e., $\mathbb{E}\left[|X-np|\right]$. 
Notice that the bound given by (\ref{markver}) is less than $\frac{1}{4}$, for any zero-mean random variable $Z$.  
This bound, being rather general, 
does not use any of the properties of the binomial law. Moreover,  it is well known 
(see Kaas and Burhman \cite{Kaas}) that a median of a binomial random variable 
of parameters $n$ and $p$ is larger than or equal to $\lfloor np \rfloor$,  
the largest integer that is less than or equal to $np$.  
Hence $\mathbb{P}\left[\text{Bin}(n,p)\geq \lfloor np \rfloor\right] \geq \frac{1}{2}$  which  suggests that, 
in the case of binomial distributions,   
there may be space for improvement upon the bound provided by (\ref{markver}).   
In a recent article 
Greenberg and Mohri \cite{Greenberg} provide the estimate 
\begin{equation}\label{mlers}
\mathbb{P}\left[\text{Bin}(n,p)\geq np\right] > \frac{1}{4}, \; \text{for}\; p\geq \frac{1}{n} . 
\end{equation}
It is remarkable that this bound, despite the fact that it is rather intuitive and has been used several times in the 
literature (see \cite{Greenberg}),  appears to have been formally verified quite recently.     
Notice that the inequality is strict. A  weaker version of   
this bound has been reported by Rigollet and Tong \cite{Rigollet} and reads as follows:
\begin{equation}\label{mlerstwo}
\mathbb{P}\left[\text{Bin}(n,p)\geq np\right] \geq \min \{p, \frac{1}{4}\}, \; \text{for} \; p \leq \frac{1}{2} . 
\end{equation}
Such bounds are of particular interest in machine learning and related areas where they are used in 
the analysis of the, so-called, relative deviation bounds and generalisation bounds. 
Briefly, both bounds (\ref{mlers}) and (\ref{mlerstwo}) are obtained using the observation 
\[ \mathbb{P}\left[\text{Bin}(n,p)\geq np\right] \geq \mathbb{P}\left[\text{Bin}(n,k/n)\geq k+1\right] , \] 
where $k$ is the unique positive integer such that $k<np \leq k+1$, and the problem is 
reduced to the one of estimating from below the tail of a binomial random variable whose 
mean is an \emph{integer}.  
In this note we apply a similar idea to the tail conditional expectation of a 
binomial random variable. This allows to obtain 
a refined version of the aforementioned bounds, when the variance is larger than $8$.   
More precisely, we have the following. \\

\begin{theorem}\label{constbound} 
Fix a positive integer $n\geq 2$ and a real number $p\in \left[\frac{1}{n}, 1-\frac{1}{n}\right]$.
Let $X\sim \text{Bin}(n,p)$. Then 
\[ \mathbb{P}\left[X \geq np\right] \geq \frac{1}{2\sqrt{2}} \cdot 
\frac{\sqrt{np(1-p)}}{1+\sqrt{np(1-p)+1}} . \]
\end{theorem}

Since the function $f(x)= \frac{\sqrt{x}}{1+\sqrt{x+1}}$ is increasing 
it is not difficult to see, by investigating the inequality $\sqrt{2x}\geq 1+\sqrt{x+1}$, that  the 
previous bound is larger than $\frac{1}{4}$, when the parameters $n,p$ satisfy $np(1-p)\geq 8$. 
We prove Theorem \ref{constbound} in the next section.  Our article ends with Section \ref{poisson} in which we 
sketch a proof of a lower bound on the probability that a Poisson random variable is 
exceeding its mean.

\section{Proof of Theorem \ref{constbound}}\label{sectiontwo}

In this section we prove our main result. We begin by collecting certain results on the 
median, the mean absolute deviation and the tail conditional expectation of a binomial random variable.  \\

\begin{lemma}\label{kaasiegel}  
Let $X\sim \text{Bin}(n,p)$ and suppose that $np$ is an \emph{integer}. Then 
\[ \mathbb{P}\left[X\geq np\right] > \frac{1}{2} . \]
\end{lemma} 
\begin{proof} This is a well known result that can be found 
in several places. See, for example, Jogdeo et al. \cite{Jogdeo}, Kaas et al. \cite{Kaas},  
or Siegel \cite[Theorem $2.2$]{Siegel}. 
\end{proof}

This result yields an upper bound on the tail conditional expectation of a binomial 
random variable whose mean is an integer. \\

\begin{lemma}\label{meaninteger} Let $X\sim \text{Bin}(n,p)$ and suppose that $np$ is an \emph{integer}. Then 
\[ \mathbb{E}\left[X | X \geq np\right] < np + \sqrt{np(1-p)}  . \]
\end{lemma}
\begin{proof} 
Clearly, we have
\[ \mathbb{E}\left[\max\{0, X-np\}\right] = \mathbb{P}\left[X \geq np\right] \cdot \mathbb{E}\left[X-np | X\geq np\right]  \] 
which, in view of Lemma \ref{kaasiegel}, implies 
\[ \mathbb{E}\left[X-np | X\geq np\right] < 2\cdot \mathbb{E}\left[\max\{0, X-np\}\right] . \]
To simplify notation,  set $Z^+ = \mathbb{E}\left[\max\{0,X-np\}\right] $ and 
$Z^- = \mathbb{E}\left[\min\{0,X-np\}\right] $. 
Then $Z^+ - Z^- = \mathbb{E}\left[X-np\right] = 0$ as well as $Z^+ +  Z^- = \mathbb{E}\left[ |X -np|\right]$
and therefore, upon adding the last two equations, we conclude 
\[ \mathbb{E}\left[\max\{0,X-np\}\right]  =\frac{1}{2} \cdot\mathbb{E}\left[ |X -np|\right] .   \]
Hence we have  
$\mathbb{E}\left[X-np | X\geq np\right] < \mathbb{E}\left[ |X -np|\right]$.
The estimate 
\[  \mathbb{E}\left[ |X -np|\right] \leq \sqrt{\mathbb{E}\left[|X-np|^2\right]} \]
finishes the proof. 
\end{proof}

Notice that the previous result employs an upper bound on the mean absolute deviation. 
The proof of Theorem \ref{constbound} will require  a  
corresponding lower bound. \\

\begin{lemma}
\label{binmeanabs} 
Fix positive integer $n\geq 2$ and let $X\sim \text{Bin}(n,p)$. 
If $p\in \left[\frac{1}{n}, 1-\frac{1}{n}\right]$  then 
\[ \mathbb{E}\left[ |X -np|\right] \geq \sqrt{\frac{np(1-p)}{2}} . \] 
\end{lemma}
\begin{proof} See Berend and Kontorovich \cite{Berend}.  
\end{proof}
  
Recall (see \cite{Shaked}) that an integer-valued random variable $X$ is said to 
be smaller than the integer-valued random variable $Y$ in the \emph{hazard rate order}, denoted $X\leq_{hr}Y$, 
if 
\[\frac{\mathbb{P}\left[X\geq k\right]}{\mathbb{P}\left[X\geq k+1\right]}\geq \frac{\mathbb{P}\left[Y\geq k\right]}{\mathbb{P}\left[Y\geq k+1\right]},\; \text{for all}\; k=1,2,\ldots. \]
Recall also that $X$ is said to be smaller than $Y$ in the \emph{likelihood ratio order}, denoted $X\leq_{lr}Y$, 
if $\frac{\mathbb{P}\left[X=k\right]}{\mathbb{P}\left[Y=k\right]}$ is decreasing in $k$. It is known (see \cite[Theorem 1.C.1]{Shaked})
that if $X\leq_{lr}Y$ then $X\leq_{hr}Y$. \\

\begin{lemma}\label{meester} Fix a positive integer $n$ and let $p,q\in (0,1)$ be such that $p<q$. 
Suppose that $X_{p} \sim \text{Bin}(n,p), X_{q} \sim \text{Bin}(n,q)$, and 
fix a positive integer $k\in \{0,1,\ldots,n\}$.  Then 
\[ \mathbb{E}\left[X_{p} | X_{p} \geq k \right]  \leq \mathbb{E}\left[X_{q} | X_{q} \geq k \right] .  \]
\end{lemma}
\begin{proof} This is a well known result (see \cite{Broman, Klenke}). We include 
some details of the proof for the sake of completeness. 
Notice that the result will follow once we show that   
\[ \mathbb{P}\left[X_{p}\geq k+t | X_{p} \geq k \right] \leq \mathbb{P}\left[X_{q}\geq k+t | X_{q} \geq k \right], \; \text{for all}\; t\in \{1,\ldots,n-k\} .\]
Fix $ t\in \{1,\ldots,n-k\}$ and note that it is enough to show that the function  
$f(p) = \frac{\mathbb{P}\left[X_p\geq k+t\right]}{\mathbb{P}\left[X_p\geq k\right]}$, where $p\in (0,1),$ 
is increasing in $p$. 
Now notice that 
\[ \frac{\mathbb{P}\left[X_p\geq k+t\right]}{\mathbb{P}\left[X_p\geq k\right]} = \prod_{j=0}^{t-1} \frac{\mathbb{P}\left[X_p\geq k+j+1\right]}{\mathbb{P}\left[X_p\geq k+j\right]}\]
which, in turn, implies that it is enough to show that, for $j=0,\ldots,t-1$, it holds  
\[ \frac{\mathbb{P}\left[X_p\geq k+j+1\right]}{\mathbb{P}\left[X_p\geq k+j\right]} \leq \frac{\mathbb{P}\left[X_q\geq k+j+1\right]}{\mathbb{P}\left[X_q\geq k+j\right]}, \; \text{for}\; p < q . \]
In other words, it is enough to show that $X_p\leq_{hr} X_q$. The later can be concluded 
either by induction on $n$ (see \cite[Proposition $1.1$]{Broman}) or from 
the fact (see \cite{Klenke}) that $X_p\leq_{lr} X_q$. 
\end{proof}

We now have all the necessary tools  to prove our main result. If $x$ is a positive real,  we denote by  
$\lceil x\rceil$ the minimum integer that is larger than or equal to $x$ and we set  $[[x]]= \lceil x\rceil -x$.\\

\begin{proof}[Proof of Theorem \ref{constbound}] 
In case  $np$ is an integer,  Lemma \ref{kaasiegel}  implies 
that  $\mathbb{P}\left[X \geq np\right] > 1/2$ and therefore the result holds true.  
So we may assume that $np$ is not an integer.      
Let $k$ be the unique positive integer such that $k<np<k+1$. 
Since $p\in \left[\frac{1}{n}, 1-\frac{1}{n}\right]$ Lemma \ref{binmeanabs} implies that 
\[ \mathbb{P}\left[X \geq np\right]=\frac{1}{2}\cdot \frac{\mathbb{E}\left[|X-np|\right]}{\mathbb{E}\left[X-np | X\geq np\right]}  \geq \frac{1}{2\sqrt{2}}\cdot \frac{\sqrt{np(1-p)}}{\mathbb{E}\left[X-np | X\geq np\right]}  \] 
and so it is enough to find an upper bound on $\mathbb{E}\left[X -np| X\geq np\right]$. 
Now notice that the assumption that $np$ is not an integer implies 
\[ \mathbb{P}\left[X\geq np\right] =   \mathbb{P}\left[X\geq k+1\right]\quad \text{as well as} \quad 
\mathbb{E}\left[X | X\geq np\right] = \mathbb{E}\left[X | X\geq k+1\right] . \] 
Let $Y \sim \text{Bin}(n, \frac{k+1}{n})$ and note that 
$Y$ is a binomial random variable whose mean is an integer. 
Since $p< \frac{k+1}{n}$, Lemma \ref{meester} yields  
\[  \mathbb{E}\left[X | X\geq k+1\right] \leq  \mathbb{E}\left[Y| Y\geq k+1\right].  \] 
Since $\mathbb{E}[Y]$ is an integer,  Lemma \ref{meaninteger} yields 
\begin{eqnarray*} 
\mathbb{E}\left[Y | Y\geq k+1\right] &\leq& k+1 + \sqrt{(k+1)\left(1-\frac{k+1}{n}\right)}  \\
&=& np + [[np]] + \sqrt{(np + [[np]]) \left(1-p-\frac{[[np]]}{n}\right)} \\
&=& np + [[np]] + \sqrt{np(1-p) +[[np]](1-2p)- \frac{[[np]]^2}{n}} \\
&\leq& np +  1 + \sqrt{np(1-p) +1-2p}  
\end{eqnarray*}
and so  
\[ \mathbb{E}\left[X -np | X\geq k+1\right] \leq   1 + \sqrt{np(1-p) +1-2p}. \]
Putting all the above together, we see that 
\[ \mathbb{P}\left[X \geq np\right] \geq  \frac{1}{2\sqrt{2}}\cdot \frac{\sqrt{np(1-p)}}{1+\sqrt{np(1-p)+1-2p}} \]
and the result follows. 
\end{proof}

\section{Concluding remarks}\label{poisson} 

So far we obtained a lower bound on the probability that a binomial random variable is exceeding its mean. 
Our proof combines the identity  
\begin{equation}\label{general} 
\frac{1}{2}\cdot\mathbb{E}\left[|X-\mu|\right] = \mathbb{P}\left[X \geq \mu\right] \cdot
\mathbb{E}\left[X-\mu | X\geq \mu\right] , \; \text{where}\; \mu = \mathbb{E}\left[X\right],
\end{equation}
with a lower bound on the mean absolute deviation (MAD) and an upper 
bound on the tail conditional expectation (TCE).   
Notice that (\ref{general}) holds true for any random variable and therefore it 
may be employed whenever one can estimate the 
MAD from below and the TCE from above.  
For example, it is known (see \cite[Example $1$]{Diaconis}) that $\mathbb{E}\left[|P_{\lambda}-\lambda|\right] = 2\lambda \frac{e^{-\lambda} \lambda^{\lfloor \lambda\rfloor}}{\lfloor \lambda\rfloor !}$, where 
$P_{\lambda}$ is a Poisson random variable of mean $\lambda$.  It is also known that when $\lambda$ 
is an integer then 
a median of a Poisson random variable is equal to its mean (see \cite[Section $2$]{Kaas}). Therefore,   
when $\lambda$ is an \emph{integer},  a similar argument as the one used in the proof of  Theorem \ref{meaninteger} yields 
\[ \mathbb{E}\left[P_{\lambda} | P_{\lambda}\geq \lambda\right] \leq \lambda + \sqrt{\lambda} .\]
When $\lambda$ is \emph{not} an integer, standard results on the likelihood 
ratio order of Poisson random variables 
(see \cite[Section $1$.C]{Shaked} and \cite{Klenke}) imply that 
\[ \mathbb{E}\left[P_{\lambda}| P_{\lambda}\geq \lambda\right] =\mathbb{E}\left[P_{\lambda}| P_{\lambda}\geq k\right] \leq   \mathbb{E}\left[P_{k}| P_{k}\geq k\right] , \]
where $k$ is the smallest integer that is larger than $\lambda$ and $P_{k}\sim \text{Poi}(k)$.  
Therefore a similar argument as the one used in the proof of Theorem \ref{constbound}, combined 
with the Stirling estimate $\lfloor\lambda\rfloor ! \leq e\lambda^{\lfloor\lambda\rfloor + \frac{1}{2}} e^{-\lfloor\lambda\rfloor}$,
yields the bound 
\[ \mathbb{P}\left[P_{\lambda}\geq \lambda\right] \geq \frac{2}{e^{\lambda-\lfloor\lambda\rfloor +1}}  \cdot \frac{\sqrt{\lambda}}{1 + \sqrt{\lambda +1}}, \; \text{for}\; P_{\lambda}\sim \text{Poi}(\lambda) . \]

Finally, let us remark that most lemmata from Section \ref{sectiontwo} appear to be extendable  
to sums of independent and heterogeneous Bernoulli random variables. 
However, we were not able to provide an analogue of Lemma \ref{binmeanabs} for this case. That is, we were unable 
to find a sharp lower estimate on $\mathbb{E}\left[| \sum_i B_i - \sum_i p_i |\right]$, where each $B_i$ is a $0/1$ Bernoulli 
random variable of mean $p_i$. Such a lower estimate could in turn provide a lower bound on  
$\mathbb{P}\left[ \sum_i B_i \geq \sum_i p_i\right]$ and we hope that we will be able to report on that matter 
in the future.

\end{document}